\documentclass{amsart}
\usepackage{amssymb} 
\usepackage{amscd}
\usepackage{url}
\usepackage{graphicx}

\numberwithin{equation}{section} 

\swapnumbers

\newtheorem{thm}{Theorem}[subsection] 
\newtheorem{lem}[thm]{Lemma}
 
\newtheorem{pro}[thm]{Proposition}
\theoremstyle{definition} 

\theoremstyle{definition} 

\allowdisplaybreaks[1]

\DeclareMathOperator{\rank}{rank}

\DeclareMathOperator{\Hom}{Hom}
\DeclareMathOperator{\Ext}{Ext}

\DeclareMathOperator{\Int}{Int}

\DeclareMathOperator{\Ad}{Ad}

\DeclareMathOperator{\Aut}{Aut}

\DeclareMathOperator{\Card}{Card}

\DeclareMathOperator{\K}{K}
\DeclareMathOperator{\codim}{codim}
\DeclareMathOperator{\RHom}{RHom}

\begin{document}
\title[Kazhdan's orthogonality conjecture]{Kazhdan's orthogonality conjecture \\for real reductive groups}

\author[J.-S. Huang]{Jing-Song Huang}
\address{Department of Mathematics, Hong Kong University of Science and
technology, Clear Water Bay, Kowloon, Hong Kong SAR, China}
\email{mahuang@ust.hk}

\author[D. Mili\v ci\' c]{Dragan Mili\v ci\'c}
\address{Department of Mathematics, University of Utah, Salt Lake City, UT 84112, USA}
\email{milicic@math.utah.edu}
  
\author [B. Sun]{Binyong Sun}
\address{Institute of Mathematics and Hua Loo-Keng Key Laboratory of Mathematics, AMSS, Chinese Academy of Sciences, Beijing, 100190, China}
\email{sun@math.ac.cn}

\thanks{J.-S.~Huang was supported by grants from
  Research Grant Council of HKSAR and National Science Foundation of China.
  B.~Sun was  supported by NSFC Grants 11525105, 11321101, and 11531008.}

\begin{abstract}
  We prove a generalization of Harish-Chandra's character orthogonality
  relations for discrete series to arbitrary Harish-Chandra modules for
  real reductive Lie groups. This result is an analogue of a
  conjecture by Kazhdan for $\mathfrak p$-adic reductive groups
  proved by Bezrukavnikov, and Schneider and Stuhler.
\end{abstract}



\maketitle

\section*{Introduction}
Let $G_0$ be a connected compact Lie group. Denote by $\mathcal
M(G_0)$ the category of finite-dimensional representations of $G_0$.
Then $\mathcal M(G_0)$ is abelian and semisimple. Denote by $\K(G_0)$
its Grothendieck group. Let $U$ and $U'$ be two finite-dimensional
representations of $G_0$. Denote by $\Hom_{G_0}(U,U')$ the complex
vector space of intertwining maps between representations $U$ and
$U'$. Then the map $(U, U') \longmapsto \dim \Hom_{G_0}(U,U')$ extends
to a biadditive pairing on $\K (G_0)$, which we call the multiplicity
pairing.

For a finite-dimensional representation $U$ of $G_0$, we denote by
$\Theta_U$ its character. Let $\mu_{G_0}$ be the normalized Haar
measure on $G_0$. Then the map
$$
(U, U') \longmapsto
\int_{G_0} \Theta_U(g) \overline{\Theta_{U'}(g)} \, d \mu_{G_0}(g)
$$
extends to another pairing on $\K (G_0)$. The Schur orthogonality
relations for characters of irreducible representations imply that
these two pairings are equal.

Let $T_0$ be a maximal torus in $G_0$. Denote by $\mathfrak g$ and
$\mathfrak t$ the complexified Lie algebras of $G_0$ and $T_0$
respectively. Let $R$ be the root system of the pair $(\mathfrak
g,\mathfrak t)$. Let $W$ be the Weyl group of $R$ and $[W]$ its order.

For any root $\alpha \in R$ define by $e^\alpha$ the
corresponding homomorphism of $T_0$ in the group of complex numbers of
absolute value $1$. Let
$$
D(t) = \prod_{\alpha \in R} (1 - e^\alpha(t))
$$
for any $t \in T_0$. Let $\mu_{T_0}$ be the normalized Haar measure
on $T_0$. Then we have the Weyl integral formula
$$
\int_{G_0} f(g) \, d\mu_{G_0}(g)
= \frac 1 {[W]} \int_{T_0}\left(\int_{G_0} f(g t g^{-1}) \, d\mu_{G_0}(g) \right)
D(t) \, d \mu_{T_0}(t)
$$
for any continuous function $f$ on $G_0$. In particular, this
implies that the above pairing is given by
$$
(U, U') \longmapsto \frac 1 {[W]}
\int_{T_0} \Theta_U(t) \overline{\Theta_{U'}(t)} D(t) \, d \mu_{T_0}(t).
$$
The equality of the above pairings was used by Hermann Weyl to
determine the formulas for the characters of irreducible
finite-dimensional representations of $G_0$ on $T_0$.

Assume now that $G_0$ is a noncompact connected semisimple Lie group
with finite center. Let $K_0$ be its maximal compact subgroup.  Assume
that the ranks of $G_0$ and $K_0$ are equal. In his work on discrete
series representations, Harish-Chandra generalized the latter
construction and defined an analogue of the pairing in this
situation. This is the elliptic pairing we discuss in Section
\ref{elliptic}. He also proved an analogue of Schur orthogonality
relations for characters of discrete series \cite{ds2}. Since the
category of square-integrable representations is semisimple, this
leads to an analogue of the above statement about equality of pairings
on its Grothendieck group.

Kazhdan discussed the elliptic pairing in the setting of
representation theory of $\mathfrak p$-adic reductive groups
\cite{kazhdan}. He conjectured an analogue of the equality of two
pairings in that setting. His conjecture was proved independently by
Bezrukavnikov \cite[Thm.~ 0.20]{bez} and Schneider and Stuler
\cite[Theorem, III.4.21]{schst}.

In this note we prove a generalization of Harish-Chandra statement for
arbitrary representations of real reductive groups. It is the exact
equivalent of Kazhdan's conjecture for real groups. The proof is
mostly formal in nature. To deal with nonsemisimplicity of the
category of representations, we replace it with its derived category.
This allows to define the analogue of the multiplicity pairing in this
setting. In Section \ref{homological}, we construct this pairing and
call it the homological pairing. Since the Grothendieck group is
generated by cohomologically induced representations, a Frobenius
reciprocity result proved in Section \ref{frobenius},
reduces the calculation of this pairing to calculation of Lie algebra
homology of nilpotent radicals of Borel subalgebras containing a
compact Cartan subalgebra in the Lie algebra of $G_0$. Finally, the
Osborne conjecture \cite{osborne} implies the equality of homological
and elliptic pairing.

Two of the authors (Huang and Sun) would like to thank Gregg Zuckerman
for some insightful comments.

\section{elliptic pairing}
\label{elliptic}
\subsection{Groups of Harish-Chandra class}
Let $G_0$ be a Lie group with finitely many connected components. Let
$\mathfrak g$ be the complexified Lie algebra of $G_0$. Assume that
$\mathfrak g$ is reductive.

Denote by $\Aut(\mathfrak g)$ the group of automorphisms of $\mathfrak
g$ and $\Ad : G_0 \longrightarrow \Aut(\mathfrak g)$ the adjoint
representation of $G_0$. Let $\Int(\mathfrak g)$ be the subgroup of
inner automorphisms.

Let $G_1$ be the derived subgroup of the identity component of $G_0$.

We say that the group $G_0$ is of {\em Harish-Chandra class} (see, for example,
\cite{ha1}, \cite[II.1]{var}) if the following properties hold:
\begin{enumerate}
\item[(HC1)]
  $\Ad(G_0) \subset \Int(\mathfrak g)$;
\item[(HC2)]
  The center of $G_1$ is finite.
\end{enumerate}

Fix a maximal compact subgroup $K_0$ of $G_0$. Let $K$ be the
compexification of $K_0$. Then $K$ is a reductive complex algebraic
group. Let $\mathfrak k \subset \mathfrak g$ be the complexified Lie
algebra of $K_0$.

\subsection{Categories of $(\mathfrak g,K)$-modules}
Fix a group $G_0$ of Harish-Chandra class and a maximal compact
subgroup $K_0$ of $G_0$. Denote by $\mathcal M(\mathfrak g, K)$ the
category of objects $(\pi, V)$ which are $\mathcal U(\mathfrak
g)$-modules and algebraic representations of $K$ on $V$, and these actions
$\pi$ are compatible, i.e.,
\begin{enumerate}
\item[(i)]
  the actions of $\mathfrak k$ as subset of $\mathcal U(\mathfrak g)$ and
  as differential of the action of $K$ agree; and
\item[(ii)]
$$
\pi(k)\pi(\xi)\pi(k^{-1}) = \pi(\Ad(k) \xi)
$$
for $k \in K$ and $\xi \in \mathfrak g$.
\end{enumerate}

The objects in this category are $(\mathfrak g,K)$-modules. The
morphisms are the linear maps intertwining the actions of $\mathcal
U(\mathfrak g)$ and $K$. For any two $(\mathfrak g,K)$-modules $U$ and
$V$ we denote by $\Hom_{(\mathfrak g,K)}(U,V)$ the complex vector
space of all morphisms of $U$ into $V$. Clearly, $\mathcal M(\mathfrak
g, K)$ is an abelian category.

We denote by $\mathcal A(\mathfrak g,K)$ the full subcategory of
$\mathcal M(\mathfrak g,K)$ consisting of all $(\mathfrak
g,K)$-modules of finite length. The objects in $\mathcal A(\mathfrak
g,K)$ are {\em Harish-Chandra modules}.

Let $V$ be a $(\mathfrak g,K)$-module. Since $K$ is reductive, $V$ is
a direct sum of irreducible finite-dimensional representations of $K$.
We say that $V$ is an {\em admissible} $(\mathfrak g,K)$-module if
$\Hom_K(F,V)$ is finite-dimensional for any finite-dimensional
irreducible representation $F$ of $K$. By a classical result of
Harish-Chandra, any irreducible $(\mathfrak g,K)$-module is
admissible. Hence, any Harish-Chandra module is admissible.

Let $V$ be a Harish-Chandra module. Denote by $V\check{\ }$ the
$K$-finite dual of $V$ equipped with the adjoint action of $\mathfrak
g$ and $K$. Then $V\check{\ }$ is the {\em dual} of $V$. The functor
$V \longmapsto V\check{\ }$ is an involutive antiequivalence of the
category $\mathcal A(\mathfrak g,K)$.

Let $\K (\mathfrak g,K)$ be the Grothendieck group of $\mathcal
A(\mathfrak g,K)$. For any $U$ in $\mathcal A(\mathfrak g,K)$, we
denote by $[U]$ the corresponding element of $\K(\mathfrak g,K)$.

To a Harish-Chandra module $V$, Harish-Chandra attaches its {\em
  character} $\Theta_V$ which is a distribution on $G_0$. The map $V
\longmapsto \Theta_V$ factors through $\K (\mathfrak g,K)$. Hence, we
can also denote by $\Theta_{[V]}$ the character of the element $[V]$
of $K(\mathfrak g,K)$. Clearly, $[V] \longmapsto \Theta_{[V]}$ is a
homomorphism of $\K(\mathfrak g,K)$ into the additive group of
distributions on $G_0$. A well known regularity theorem of
Harish-Chandra states that the distribution $\Theta_{[V]}$ is given by
a locally integrable function which is real analytic on the set of
regular elements in $G_0$. Abusing the notation, we denote it by the
same letter. More precisely we have
$$
\Theta_{[V]}(f) = \int_{G_0} \Theta_{[V]}(g) f(g) \, d\mu_{G_0}(g)
$$
for any compactly supported smooth function $f$ on $G_0$.

\subsection{Weyl integral formula for the elliptic set}
Assume that the rank of $G_0$ is equal to the rank of $K_0$. Let $T_0$
be a Cartan subgroup of $K_0$. Then $T_0$ is also a Cartan subgroup in
$G_0$. An element $g \in G_0$ is {\em elliptic} if $\Ad(g)$ is
semisimple and its eigenvalues are complex numbers of absolute value
$1$. Denote by $E$ the set of all regular elliptic elements in
$G_0$. Also we denote by $T_0'$ the set of regular elements in
$T_0$. Clearly, $E$ is an open set in $G_0$, invariant under
conjugation by elements of $G_0$ and every conjugacy class in $E$
intersects $T_0'$.  Let $\mu_{G_0}$ be a Haar measure on $G_0$. Then
there exists a unique positive measure $\nu$ on $T_0$ such that
$$
\int_E f(g) \, d\mu_{G_0}(g)
= \int_{T_0} \left( \int_{G_0} f(g t g^{-1}) \, d \mu_{G_0}(g) \right) \, d\nu(t)
$$
for any compactly supported continuous function $f$ on $G_0$. It is
evident that the measure $\nu$ does not depend on the choice of Haar
measure $\mu_{G_0}$.

Let $\mathfrak t$ be the complexified Lie algebra of $T_0$. Denote by
$R$ the root system of $(\mathfrak g,\mathfrak t)$ in $\mathfrak
t^*$. For any root $\alpha$ in $R$ we denote by $e^\alpha$ the
corresponding homomorphism of $T_0$ into the group of complex numbers
of absolute value $1$.

We put
$$
D(t) = \prod_{\alpha \in R} (1 - e^\alpha(t))
$$
for $t \in T_0$. Clearly, $D$ is a positive real analytic function
on $T_0$.

The normalizers of $T_0$ in $G_0$ and $K_0$ are equal and we denote
them by $N(T_0)$. The quotient $W_0 = N(T_0)/T_0$ is naturally
identified with a subgroup $W_0$ of the Weyl group $W$ of the root
system $R$ (cf.~\cite[Part II, Sec.~1]{var}). Denote by $[W_0]$
the order of $W_0$.

Then we have the following formula \cite[Part II, Sec.~15, Lemma 17]{var}
$$
d \nu(t) = \frac 1 {[W_0]} D(t) \, d \mu_{T_0}(t),
$$
where $\mu_{T_0}$ is the normalized Haar measure on the group
$T_0$.

\subsection{Elliptic pairing}
Still assuming that the ranks of $G_0$ and $K_0$ are equal, let $R^+$
be a set of positive roots in $R$. Harish-Chandra proved that the
function
$$
\Psi_{[V]}(t)
= \prod_{\alpha \in R^+}(1 - e^\alpha(t)) \Theta_{[V]}(t), \quad t \in T_0',
$$
extends to a real analytic function on $T_0$ \cite{ds2}.

Therefore, for any two elements $[U]$ and $[V]$ in $\K (\mathfrak
g,K)$, we can define
\begin{multline*}
\langle [U] \mid [V] \rangle_{ell}
= \int_{T_0} \Theta_{[U]}(t) \overline{\Theta_{[V]}(t)} \, d \nu(t) \\
= \frac 1 {[W_0]} \int_{T_0} D(t) \Theta_{[U]}(t)
\overline {\Theta_{[V]}(t)} \, d \mu_{T_0}(t) 
= \frac 1 {[W_0]} \int_{T_0}  \Psi_{[U]}(t)
\overline {\Psi_{[V]}(t)} \, d \mu_{T_0}(t).
\end{multline*}
This is clearly a biadditive pairing on $\K (\mathfrak g,K)$ with
values in $\mathbb C$, which we call the {\em elliptic pairing}.

If the group $G_0$ has rank greater than its maximal compact subgroup $K_0$,
we define the elliptic pairing on $\K (\mathfrak g,K)$ as the zero
pairing.

\section{Homological pairing}
\label{homological}
\subsection{Derived categories of $(\mathfrak g,K)$-modules}
It is well known that $\mathcal M(\mathfrak g,K)$ contains enough
injective and projective objects \cite[Ch.~I]{bw}. Moreover, for any
two $(\mathfrak g,K)$-modules $U$ and $V$ we have $\Ext_{(\mathfrak
  g,K)}^p(U,V) = 0$ for $p > \dim(\mathfrak g/\mathfrak k)$.

Denote by $D^b(\mathfrak g,K)$ the bounded derived category of
$\mathcal M(\mathfrak g, K)$ and $D^b(\mathfrak g,K)^\circ$ its
opposite category. Then we have the derived bifunctor
$\RHom_{(\mathfrak g,K)}$ from $D^b(\mathfrak g,K)^\circ \times
D^b(\mathfrak g,K)$ into the bounded derived category $D^b(\mathbb C)$
of complex vector spaces. As it is well known (see, for example,
\cite[Thm.~13.4.1]{ks})
\begin{equation}
  \label{rhom}
H^p(\RHom_{(\mathfrak g,K)}(U^\cdot, V^\cdot))
= \Hom_{D^b(\mathfrak g,K)}(U^\cdot, V^\cdot[p])
\end{equation}
for any two complexes $U^\cdot$, $V^\cdot$ in $D^b(\mathfrak g,K)$.

\subsection{A finiteness result}
Denote by $D_{adm}^b(\mathfrak g,K)$ the full subcategory of
$D^b(\mathfrak g,K)$ consisting of complexes with cohomology in
$\mathcal A(\mathfrak g,K)$. Then $D_{adm}^b(\mathfrak g,K)$ is a
triangulated category with natural $t$-structure and core $\mathcal
A(\mathfrak g,K)$.

Let $D : \mathcal A(\mathfrak g,K) \longrightarrow D_{adm}^b(\mathfrak
g,K)$ be the natural map attaching to a module $U$ the complex
$D(U)^\cdot$ such that $D(U)^0 = U$ and $D(U)^p = 0$ for $p \ne 0$.

\begin{lem}
  Let $U^\cdot$ and $V^\cdot$ be two objects in $D_{adm}^b(\mathfrak g,K)$.
  Then $\RHom_{(\mathfrak g,K)}(U^\cdot, V^\cdot)$ is a bounded complex of
  complex vector spaces with finite-dimensional cohomology.
\end{lem}

\begin{proof}
Let $U$ and $V$ be two objects in $\mathcal A(\mathfrak g,K)$. Then
they are admissible. Hence, by \cite[I.2.8]{bw}, $\Ext_{(\mathfrak
  g,K)}^p(U,V)$ are finite-dimensional for any $p \in \mathbb Z_+$.

Therefore, $\Hom_{D^b(\mathfrak g,K)}(D(U)^\cdot,D(V)^\cdot[p])$ is
finite-dimensional for any two modules $U$ and $V$ in $\mathcal
A(\mathfrak g,K)$ and $p \in \mathbb Z$. By induction on the
cohomological length of a bounded complex using standard truncation
arguments (cf.~\cite[Ch.~3, 4.2]{dercat}), this implies
that $\Hom_{D^b(\mathfrak g,K)}(U^\cdot,V^\cdot)$ is
finite-dimensional for any two bounded complexes $U^\cdot$ and
$V^\cdot$ in $D_{adm}^b(\mathfrak g,K)$.  By \eqref{rhom}, this
implies the statement of the lemma.
\end{proof}

Therefore, we can consider the bifunctor $\RHom_{(\mathfrak g,K)}$
from $D_{adm}^b(\mathfrak g,K)^o \times D_{adm}^b(\mathfrak g,K)$ into
the full subcategory $D_{fd}^b(\mathbb C)$ of $D^b(\mathbb C)$
consisting of complexes with finite-dimensional cohomology.

\subsection{Homological pairing}
Since the category $\mathcal A(\mathfrak g,K)$ is not semisimple, to
define a natural pairing on its Grothendieck group $\K (\mathfrak
g,K)$ we have to use homological algebra.

We identify the Grothendieck group of the triangulated category
$D_{adm}^b(\mathfrak g,K)$ with $\K(\mathfrak g,K)$ via the map
$[U^\cdot] \longmapsto \sum_{p \in \mathbb Z} (-1)^p [H^p(U^\cdot)]$
(see, for example, \cite[Ch.~4, Sec.~3.5]{dercat}). In the same
fashion, the Grothendieck group of $D_{fd}^b(\mathbb C)$ is identified
with integers $\mathbb Z$ via the map $[A^\cdot] \longmapsto \sum_{p
  \in \mathbb Z} (-1)^p \dim H^p(A^\cdot)$.

Composition of the map $\RHom_{(\mathfrak g,K)} : D_{adm}^b(\mathfrak
g,K)^o \times D_{adm}^b(\mathfrak g,K) \longrightarrow
D_{fd}^b(\mathbb C)$ with the natural map of $D_{fd}^b(\mathbb C)
\longrightarrow \K(D_{fd}^b(\mathbb C)) = \mathbb Z$ factors through
$\K(\mathfrak g,K) \times \K(\mathfrak g,K)$. Hence, it defines a biadditive
pairing $\K(\mathfrak g,K) \times \K(\mathfrak g,K) \longrightarrow
\mathbb Z$. We call it the {\em homological pairing} on $\K(\mathfrak
g,K)$. For $U^\cdot$ and $V^\cdot$ in $D_{adm}^b(\mathfrak g,K)$ we
denote the value of this pairing by $\langle [U^\cdot] \mid [V^\cdot]
\rangle_{(\mathfrak g,K)}$.

\begin{pro}\footnote{Because of this result, this pairing is sometimes called
    the {\em Euler-Poincar\'e pairing}.}
Let $U$ and $U'$ be two modules in $\mathcal A(\mathfrak g,K)$. Then
we have
  $$
  \langle [U] \mid [U'] \rangle_{(\mathfrak g,K)}
  = \sum_{p \in \mathbb Z} (-1)^p \dim \Ext_{(\mathfrak g,K)}^p(U,U').
  $$
\end{pro}

\begin{proof}
By \eqref{rhom}, we have
\begin{multline*}
\langle [U] \mid [U'] \rangle_{(\mathfrak g,K)} =
  \langle [D(U)^\cdot] \mid [D(U')^\cdot] \rangle_{(\mathfrak g,K)} \\
  = \sum_{p \in \mathbb Z} (-1)^p
  \dim H^p(\RHom_{(\mathfrak g,K)}(D(U)^\cdot,D(U')^\cdot)) \\
  = \sum_{p \in \mathbb Z} (-1)^p
\dim \Hom_{D^b(\mathfrak g,K)}(D(U)^\cdot,D(U')^\cdot[p]) 
= \sum_{p \in \mathbb Z} (-1)^p \dim \Ext_{(\mathfrak g,K)}^p(U,U').
\end{multline*}
\end{proof}

\section{Frobenius reciprocity}
\label{frobenius}
\subsection{Frobenius reciprocity for cohomological induction}
Let $\sigma$ be the Cartan involution corresponding to the maximal
compact subgroup $K_0$ of $G_0$. Let $\mathfrak c$ be a
$\sigma$-stable Cartan subalgebra of $\mathfrak g$. Then $\mathfrak c
= \mathfrak t \oplus \mathfrak a$ is the decomposition into
eigenspaces for eigenvalues $1$ and $-1$ of $\sigma$. Let $T$ be the
subgroup of $K$ which centralizes $\mathfrak c$. Then its Lie algebra
is identified with $\mathfrak t$.

As before, we define the category $\mathcal M(\mathfrak c, T)$
consisting of $(\mathfrak c,T)$-modules. Clearly, an irreducible
$(\mathfrak c,T)$-module is finite-dimensional. Therefore, in this
case $\mathcal A(\mathfrak c, T)$ is the full subcategory of
finite-dimensional $(\mathfrak c,T)$-modules.

Let $R$ be the root system of the pair $(\mathfrak g, \mathfrak c)$ in
$\mathfrak c^*$. Denote by $R^+$ a set of positive roots in $R$.
Let
$$
\mathfrak n = \bigoplus_{\alpha \in R^+} \mathfrak g_\alpha
$$
and
$$
\mathfrak b = \mathfrak c \oplus \mathfrak n.
$$
Then $\mathfrak b$ is a Borel subalgebra in $\mathfrak
g$. Moreover, $T$ normalizes $\mathfrak b$.

Let $U$ be a $(\mathfrak g, K)$-module. Then the zero Lie algebra
homology $H_0(\mathfrak n,U)$ is a $(\mathfrak c, T)$-module.  By
abuse of notation, we denote by $H_\bullet(\mathfrak n, - )$ the
derived functor of Lie algebra homology from $D^b(\mathfrak g, K)$
into $D^b(\mathfrak c,T)$. Hence, the $p$-th Lie algebra homology
group $H_p(\mathfrak n, U)$ of $U$ is $H^{-p}(H_\bullet(\mathfrak n,
D(U)^\cdot))$.

We consider the forgetful functor from $\mathcal M(\mathfrak g,K)$
into $\mathcal M(\mathfrak g,T)$. It has a right adjoint
$\Gamma_{K,T}$ -- the {\em Zuckerman functor} from $\mathcal
M(\mathfrak g,T)$ into $\mathcal M(\mathfrak g,K)$. Its right
cohomological dimension is $\le \dim(K/T)$. We follow the forgetful
functor by the forgetful functor from $\mathcal M(\mathfrak g,T)$ into
$\mathcal M(\mathfrak b,T)$.  This functor also has a right adjoint
functor $P$ constructed as follows.  Consider $\mathcal U(\mathfrak
g)$ as a $\mathcal U(\mathfrak b)$-module for left multiplication. Let
$V$ be a $(\mathfrak b,T)$-module. Then $\Hom_{\mathcal U(\mathfrak
  b)}(\mathcal U(\mathfrak g),V)$ has a natural $T$-action, $T$ acting
on $\mathcal U(\mathfrak g)$ via the adjoint action. Let
$\Hom_{\mathcal U(\mathfrak b)}(\mathcal U(\mathfrak g),V)_{[T]}$ be
the largest algebraic submodule of $\Hom_{\mathcal U(\mathfrak
  b)}(\mathcal U(\mathfrak g),V)$ for that action of $T$. Then
$\mathcal U(\mathfrak g)$ acts on this module by right multiplication
on $\mathcal U(\mathfrak g)$. In this way, one gets the $(\mathfrak
g,T)$-module $P(V)$. The functor $P : \mathcal M(\mathfrak b,T)
\longrightarrow \mathcal M(\mathfrak g,T)$ is exact.

Consider now a $(\mathfrak c,T)$-module $V$. We can view it as a
$(\mathfrak b,T)$-module. This functor has a left adjoint functor
$H_0(\mathfrak n,-)$.

We define the functor
$$
I(\mathfrak c, R^+, - ) : \mathcal M(\mathfrak c,T) \longrightarrow
\mathcal M(\mathfrak g,K)
$$
as the composition of the functor $P$ followed by the Zuckerman functor
$\Gamma_{K,T}$.

The next result is a formal consequence of the above discussion.

\begin{lem}
  The functor $I(\mathfrak c,R^+, - )$ is a right adjoint of the
  functor $H_0(\mathfrak n, - )$.
\end{lem}

The right derived functors $R^p I(\mathfrak c,R^+, - ) : \mathcal
M(\mathfrak c,T) \longrightarrow \mathcal M(\mathfrak g,K)$ are called
the {\em cohomological induction} functors.

Since both functors $I(\mathfrak c,R^+, - )$ and $H_0(\mathfrak n, -)$
have finite cohomological dimension, their derived functors exist as
functors between corresponding bounded derived categories, and a formal
consequence of the above lemma is the following version of Frobenius
reciprocity \cite[Ch.~5, Thm.~1.7.1]{dercat}.

\begin{pro}
  \label{fro}
  The right derived functor $RI(\mathfrak c,R^+, - ) : D^b(\mathfrak
  c,T) \longrightarrow D^b(\mathfrak g,K)$ is a right adjoint of
  $H_\bullet(\mathfrak n, - )$.
\end{pro}

\subsection{Finiteness results}
We also have the following finiteness results.

\begin{lem}
  \label{fin1}
  Let $U^\cdot$ be an object in $D_{adm}^b(\mathfrak g,K)$. Then
  $H_\bullet(\mathfrak n, U^\cdot)$ is an object in $D_{adm}^b(\mathfrak c,T)$.
\end{lem}

\begin{proof}
By induction in homological length and standard truncation argument
(cf.~\cite[Ch.~3, 4.2]{dercat}), we can reduce the proof
to the case $U^\cdot = D(U)^\cdot$, where $U$ is a $(\mathfrak
g,K)$-module of finite length. In this case, it is enough to prove
that Lie algebra homology groups $H_p(\mathfrak n,U)$, $p \in \mathbb
Z_+$, are finite-dimensional. This is well-known, a geometric proof
can be found, for example, in \cite[Ch.~4, Thm.~4.1]{book}.
\end{proof}

\begin{lem}
  \label{fin2}
 Let $V^\cdot$ be an object in $D_{adm}^b(\mathfrak c,T)$. Then
  $RI(\mathfrak c, R^+, V^\cdot)$ is an object in $D_{adm}^b(\mathfrak g,K)$.
\end{lem}

\begin{proof}
  By induction in homological length and standard truncation argument
  (cf.~\cite[Ch.~3, 4.2]{dercat}), we can reduce the
  proof to the case $V^\cdot = D(V)^\cdot$, where $V$ is a
  finite-dimensional $(\mathfrak c,T)$-module. Then, by induction in
  dimension, we can reduce the proof to the case of finite-dimensional
  irreducible $(\mathfrak c,T)$-modules. In this case, by the main
  result of \cite{hmsw1}, cohomologies of the complex $RI(\mathfrak c,
  R^+, D(V)^\cdot)$ are duals of cohomologies of holonomic $\mathcal
  D$-modules on the flag variety $X$ of $\mathfrak g$.  By
  \cite[Ch.~3, Thm.~6.3]{book}, these are $(\mathfrak g,K)$-modules of
  finite length.
\end{proof}

\subsection{Homological pairing and cohomological induction}
Frobenius reciprocity and the finiteness results \ref{fin1} and
\ref{fin2} imply the following version of Frobenius reciprocity for
the homological pairings on the Grothendieck groups $\K(\mathfrak
g,K)$ and $\K(\mathfrak c,T)$.

\begin{pro}
\label{fro-pairing}
  Let $V^\cdot$ be an object in $D_{adm}^b(\mathfrak c,T)$ and $U^\cdot$
  an object in $D_{adm}^b(\mathfrak g,K)$. Then we have
$$
\langle [U^\cdot] \mid [RI(\mathfrak c,R^+, V^\cdot)] \rangle_{(\mathfrak g,K)}
= \langle [H_\bullet(\mathfrak n,U^\cdot)] \mid [V^\cdot] \rangle_{(\mathfrak c,T)}.
$$
\end{pro}

\begin{proof}
 Using (\ref{rhom}) twice, we have
\begin{multline*}
  \langle [U^\cdot] \mid [RI(\mathfrak c,R^+, V^\cdot)] \rangle_{(\mathfrak g,K)}
  = \sum_{p \in \mathbb Z} (-1)^p \dim H^p(\RHom_{(\mathfrak g,K)} (U^\cdot,
  RI(\mathfrak c,R^+, V^\cdot))) \\
  = \sum_{p \in \mathbb Z} (-1)^p \dim \Hom_{D^b(\mathfrak g,K)} (U^\cdot,
  RI(\mathfrak c,R^+, V^\cdot)[p])\\
 = \sum_{p \in \mathbb Z} (-1)^p \dim \Hom_{D^b(\mathfrak g,K)} (U^\cdot,
  RI(\mathfrak c,R^+, V^\cdot[p]))\\
  = \sum_{p \in \mathbb Z} (-1)^p \dim \Hom_{D^b(\mathfrak c,T)} (H_\bullet(\mathfrak n,
  U^\cdot), V^\cdot[p])\\
  = \sum_{p \in \mathbb Z} (-1)^p \dim H^p (\RHom_{(\mathfrak c,T)}
  (H_\bullet(\mathfrak n,U^\cdot), V^\cdot))
  = \langle [H_\bullet(\mathfrak n,U^\cdot)] \mid [V^\cdot] \rangle_{(\mathfrak c,T)}.
\end{multline*}
\end{proof}

\section{Calculation of homological pairing}

\subsection{A vanishing result}
Let $\mathfrak a$ be a nonzero Lie algebra. Denote by $\mathcal M(\mathfrak
a)$ the category of $\mathfrak a$-modules. 

The following vanishing result is well-known, we include a proof for
convenience of the reader.

\begin{lem}
  \label{lavan}
  Let $U$ and $V$ be two finite-dimensional $\mathfrak a$-modules.
  Then
  $$
  \sum_{p=0}^{\dim \mathfrak a} (-1)^p \dim \Ext_{\mathfrak a}^p (U,V) = 0.
  $$
\end{lem}

\begin{proof}
  By \cite[Ch.~I]{bw}, we have
  $$
\Ext_{\mathfrak a}^p (U,V) = H^p(\mathfrak a, \Hom_{\mathbb C}(U,V))
$$
for any $p \in \mathbb Z_+$. Therefore, we have
$$
\sum_{p=0}^{\dim \mathfrak a} (-1)^p \dim \Ext_{\mathfrak a}^p (U,V)
= \sum_{p=0}^{\dim \mathfrak a} (-1)^p
 \dim H^p(\mathfrak a, \Hom_{\mathbb C}(U,V)).
 $$
 On the other hand, for any finite-dimensional representation $F$ of
 $\mathfrak a$, we have
\begin{multline*}
\sum_{p=0}^{\dim \mathfrak a} (-1)^p \dim H^p(\mathfrak a,F) =
\sum_{p=0}^{\dim \mathfrak a} (-1)^p \dim \Hom_{\mathbb C}(\bigwedge^p \mathfrak a, F) \\
= \left(\sum_{p=0}^{\dim \mathfrak a} (-1)^p \dim \bigwedge^p \mathfrak a \right)
\cdot \dim  F
\end{multline*}
  using the standard complex of Lie algebra cohomology.
  Finally, we have
  $$
  \sum_{p=0}^{\dim \mathfrak a} (-1)^p \dim \bigwedge^p \mathfrak a =
  \sum_{p=0}^{\dim \mathfrak a} (-1)^p \binom  {\dim \mathfrak a} p
  = (1 - 1)^{\dim \mathfrak a} = 0,
  $$
what implies our assertion.
\end{proof}

Consider now the homological pairing on a $\sigma$-stable Cartan
subalgebra $\mathfrak c$. Assume that $\mathfrak a \ne 0$. Let $V$
and $V'$ be two irreducible finite-dimensional $(\mathfrak
c,T)$-modules. Then, by Schur lemma, $\mathfrak a$ acts on $V$ and
$V'$ by linear forms $\mu, \mu' \in \mathfrak a^*$. The restrictions of
$V$ and $V'$ to $(\mathfrak t,T)$ are irreducible modules which we
denote by the same symbol. By \cite[Ch.~I]{bw}, we have
\begin{multline*}
\Ext_{(\mathfrak c,T)}^n (V,V') = \bigoplus_{p+q = n} \Ext_{(\mathfrak t,T)}^p (V,V')
\otimes \Ext_{\mathfrak a}^q (\mathbb C_{\mu},\mathbb C_{\mu'}) \\
= \Hom_{(\mathfrak t,T)}(V,V') \otimes
\Ext_{\mathfrak a}^n (\mathbb C_{\mu},\mathbb C_{\mu'}),
\end{multline*}
 for any $n \in \mathbb Z_+$, since $\mathcal A(\mathfrak t,T)$ is
 semisimple. Therefore, by \ref{lavan}, we have
$$
\langle [V], [V'] \rangle_{(\mathfrak c,T)} =
\dim \Hom_{(\mathfrak t,T)}(V,V') \cdot \sum_{p=0}^{\dim \mathfrak a} (-1)^p
\Ext_{\mathfrak a}^p (\mathbb C_{\mu},\mathbb C_{\mu'}) = 0.
$$

Hence, we get the following elementary vanishing result.

\begin{lem}
  \label{van}
If $\mathfrak a \ne 0$, the homological pairing on $\K (\mathfrak
c,T)$ is zero.
\end{lem}

\subsection{Localization and Grothendieck groups of Harish-Chandra modules}
To calculate the homological pairing on $\K(\mathfrak g,K)$ we have to
invoke the geometric classification of irreducible Harish-Chandra
modules. We use freely the notation from \cite{hmsw1} and
\cite{penrose}.

Let $\mathcal Z(\mathfrak g)$ be the center of the enveloping algebra
$\mathcal U(\mathfrak g)$ of $\mathfrak g$. Using Harish-Chandra
homomorphism, the maximal ideals in $\mathcal Z(\mathfrak g)$
correspond to the orbits of the Weyl group $W$ in the abstract Cartan
algebra $\mathfrak h$ of $\mathfrak g$. For an orbit $\theta$ we
denote by $I_\theta$ the corresponding maximal ideal in $\mathcal
Z(\mathfrak g)$. Let $\mathcal U_\theta$ be the quotient of
$\mathcal U(\mathfrak g)$ by the two-sided ideal generated by
$I_\theta$. We denote by $\mathcal A(\mathcal U_\theta,K)$ the full
subcategory of $\mathcal A(\mathfrak g,K)$ consisting of
Harish-Chandra modules with infinitesimal character corresponding to
$I_\theta$. Let $\K(\mathcal U_\theta,K)$ be the Grothendieck group of
$\mathcal A(\mathcal U_\theta,K)$. Then, we have the direct sum
decomposition
$$
\K(\mathfrak g,K) = \bigoplus_{\theta} \K(\mathcal U_\theta,K).
$$
By Wigner's lemma \cite[Ch.~I]{bw}, the subgroups $\K(\mathcal
U_\theta,K)$ are mutually orthogonal with respect to the homological
pairing.  Therefore, we have to calculate it on these subgroups only.

Fix a Weyl group orbit $\theta$. Then there exists a $\lambda$ in this
orbit which is antidominant. Let $\mathcal M_{coh}(\mathcal
D_\lambda,K)$ be the category of coherent $K$-equivariant $\mathcal
D_\lambda$-modules on the flag variety $X$ of $\mathfrak g$.  The
objects of $\mathcal M_{coh}(\mathcal D_\lambda,K)$ are called {\em
  Harish-Chandra sheaves}. Since Harish-Chandra sheaves are holonomic
\cite[Thm.~6.1]{penrose}, they are of finite length.

The functor of global sections $\Gamma(X,-)$ is an exact functor from
the category $\mathcal M_{coh}(\mathcal D_\lambda,K)$ into $\mathcal
A(\mathcal U_\theta,K)$ since $\lambda$ is antidominant. More
precisely, $\mathcal A(\mathcal U_\theta,K)$ is equivalent to a
quotient category of $\mathcal M_{coh}(\mathcal D_\lambda,K)$
(cf.~\cite[3.8]{penrose}).  Let $\K(\mathcal D_\lambda,K)$ be the
Grothendieck group of $\mathcal M_{coh}(\mathcal D_\lambda,K)$. The
above statement implies that $\K(\mathcal U_\theta,K)$ is a quotient
group of $\K (\mathcal D_\lambda,K)$.

 It is easy to describe all irreducible Harish-Chandra sheaves. The group
 $K$ has finitely many orbits in $X$. Let $Q$ be a $K$-orbit in
 $X$. There exist a finite family of $K$-equivariant connections on
 $Q$ compatible with $\lambda + \rho$. For such connection $\tau$, we
 denote by $\mathcal I(Q,\tau)$ the $\mathcal D$-module direct image
 of $\tau$ under the natural inclusion of $Q$ into $X$. Then $\mathcal
 I(Q,\tau)$ is the {\em standard Harish-Chandra sheaf} attached to the
 geometric data $(Q,\tau)$.  It has a unique irreducible subobject
 $\mathcal L(Q,\tau)$. All irreducible objects in $\mathcal
 M_{coh}(\mathcal D_\lambda,K)$ are isomorphic to $\mathcal L(Q,\tau)$
 for some geometric data $(Q,\tau)$.

Therefore, the classes $[\mathcal L(Q,\tau)]$ form a basis of
$\K(\mathcal D_\lambda,K)$. Since the other composition factors of
$\mathcal I(Q,\tau)$ correspond to $K$-orbits in the boundary of $Q$,
we immediately see that classes of $[\mathcal I(Q,\tau)]$ also form a
basis of $\K(\mathcal D_\lambda,K)$.

We call $\Gamma(X,\mathcal I(Q,\tau))$ the {\em standard
  Harish-Chandra module} attached to geometric data $(Q,\tau)$. The
above discussion implies that classes of standard
Harish-Chandra modules, for all geometric data $(Q,\tau)$, generate
$\K(\mathcal U_\theta,K)$.

Let $\theta\check{\ }$ be the orbit of $-\lambda$. Then the duality $U
\longmapsto U\check{\ }$ is an antiequivalence of the category
$\mathcal A(\mathcal U_\theta,K)$ with $\mathcal A(\mathcal
U_{\theta\check{\ }},K)$. Therefore, the classes $[\Gamma(X, \mathcal
  I(Q,\tau))\check{\ }]$ generate $\K(\mathcal U_{\theta\check{\ }},K)$.

Let $x$ be point in $Q$. Let $\mathfrak b_x$ be corresponding Borel
subalgebra in $\mathfrak g$. It contains a $\sigma$-stable Cartan
subalgebra $\mathfrak c$ and all such Cartan subalgebras are
$K$-conjugate \cite [Lemma 5.3]{penrose}. Let $\mathfrak n =
[\mathfrak b_x,\mathfrak b_x]$.  Denote by $R^+$ the set of positive
roots determined by $\mathfrak n$.

Let $V$ be the irreducible representation of $(\mathfrak c,T)$ on the
geometric fiber $T_x(\tau)$ of $\tau$ at $x$. Let $\Omega_X$ be the
cotangent bundle of $X$ and $T_x(\Omega_X)$ its geometric fiber at
$x$. Then the duality theorem \cite[Thm.~4.3]{hmsw1} states that
$$
\Gamma(X,\mathcal I(Q,\tau))\check{\ } = R^sI(\mathfrak c, R^+,
V\check{\ } \otimes T_x(\Omega_X))
$$
for $s = \dim (\mathfrak k \cap \mathfrak n)$. In addition,
$R^pI(\mathfrak c, R^+, V\check{\ } \otimes T_x(\Omega_X)) = 0$ for $p
\ne s$.\footnote{Actually, \cite[Thm.~4.3]{hmsw1} assumes that $G_0$
  is connected semisimple Lie group. In \cite[Appendix B]{hmsw1}
  there is an explanation how to extend it to all groups in the
  Harish-Chandra class.}

By \ref{fro-pairing}, this immediately implies that
\begin{multline}
  \label{calc1}
  \langle [U] \mid [\Gamma(X,\mathcal I(Q,\tau))\check{\ }]
  \rangle_{(\mathfrak g,K)} =
  \langle [U] \mid [R^sI(\mathfrak c,R^+,V\check{\ } \otimes T_x(\Omega_X))]
\rangle_{(\mathfrak g,K)} \\
= (-1)^s \langle [D(U)^\cdot] \mid [RI(\mathfrak c,R^+, D(V\check{\ }
  \otimes T_x(\Omega_X))^\cdot)]\rangle_{(\mathfrak g,K)} \\
= (-1)^s \langle [H_\bullet(\mathfrak n, D(U)^\cdot)] \mid D(V\check{\ }
\otimes T_x(\Omega_X))^\cdot]\rangle_{(\mathfrak c,T)} \\
  =
  (-1)^s \sum_{p \in \mathbb Z} (-1)^p \langle [H_p(\mathfrak n, U)] \mid
  [V\check{\ }
\otimes T_x(\Omega_X)]\rangle_{(\mathfrak c,T)} 
\end{multline}
for any $U$ in $\mathcal A(\mathfrak g,K)$. Hence, by \ref{van}, the
homological pairing on $\K(\mathfrak g,K)$ vanishes if the orbit $Q$ in
the second variable is attached to a Cartan subalgebra $\mathfrak c$
with $\mathfrak a \ne \{0\}$.

\subsection{Unequal rank case}
If $\rank G_0 > \rank K_0$, any $\sigma$-stable Cartan subalgebra in
$\mathfrak g$ has $\mathfrak a \ne \{0\}$. Hence, we see that the
homological pairing vanishes on $\K(\mathfrak g,K)$ i.e., we have the
following generalization of \ref{van}.

\begin{thm}
\label{van2}
  If $\rank G_0 > \rank K_0$, the homological pairing vanishes on $\K
  (\mathfrak g,K)$.
\end{thm}

\subsection{Symmetry of Euler characteristic of Lie algebra homology}
It remains to treat the case $\rank G_0 = \rank K_0$. In this case,
the group $G_0$ has a compact Cartan subgroup contained in $K_0$.  All
such Cartan subgroups are conjugate by $K_0$.

We fix a compact Cartan subgroup $T_0$. We denote by $\mathfrak t$ the
complexification of its Lie algebra and by $T$ the complexification of
$T_0$. We denote by $R$ the root system of $(\mathfrak g,\mathfrak t)$
in $\mathfrak t^*$.

Clearly, the category $\mathcal A(\mathfrak t,T)$ is just the category
of finite-dimensional algebraic representations of $T$, hence it is
semisimple. In addition, $\K (\mathfrak t,T)$ is a ring with
multiplication given by $[V] \cdot [V'] = [V \otimes_{\mathbb C} V']$
for finite-dimensional algebraic representations $V$ and $V'$ of $T$.
The ring $\K(\mathfrak t,T)$ contains as a subring the additive
subgroup generated by all characters $e^\mu$ of $T$ where $\mu$ is a
weight in the root lattice of $R$. Moreover, the subgroup $W_0$ of the
Weyl group $W$ acts naturally on $\K (\mathfrak t,T)$.

Clearly, the homological pairing on $\K (\mathfrak t,T)$ is invariant
for the action of $W_0$.  Moreover, it is invariant for multiplication
by $e^\mu$ for any weight $\mu$, i.e., we have
$$
\langle A \cdot e^\mu \mid B \cdot e^\mu \rangle_{(\mathfrak t,T)} =
\langle A \mid B \rangle_{(\mathfrak t,T)}
$$
for any $A$, $B$ in $\K(\mathfrak t,T)$.

Let $R^+$ be a set of positive roots in $R$. Let $\rho$ be the half
sum of roots in $R^+$. Denote by $\mathfrak n$ the nilpotent Lie
algebra spanned by root subspaces $\mathfrak g_\alpha$ for roots
$\alpha \in R^+$.

Our calculation is based on remarkable symmetry properties of the
Euler characteristic of Lie algebra homology for $\mathfrak n$ of
Harish-Chandra modules. They follow from the Osborne conjecture
\cite{osborne}.\footnote{Actually, we need just a special case for
  compact Cartan subgroups \cite[7.27]{osborne}.}  Let $U$ be a
Harish-Chandra module and let $\Theta_U$ be its character. By the
Osborne conjecture, we have
$$
\Theta_U = \frac { \sum_{p \in \mathbb Z} (-1)^p \Theta_{H_p(\mathfrak n,U)}}
  {\prod_{\alpha \in R^+} (1 - e^{\alpha})}
$$
on the regular elements $T_0'$ in the compact Cartan
subgroup $T_0$.

First we need a simple symmetry property of the denominator in this
formula. Let $w \in W$, then $\rho - w\rho$ is a sum of all roots in
$R^+ \cap (-wR^+)$, hence it defines a character $e^{w\rho - \rho}$ of
$T_0$.

 We denote by $\epsilon$ the sign representation of $W$.

\begin{lem}
  \label{weyldenom}
  For any $w \in W$ we have
$$
\prod_{\alpha \in wR^+} (1 - e^{\alpha}) = \epsilon(w) e^{w \rho - \rho}
    \prod_{\alpha \in R^+} (1 - e^{\alpha}).
$$
\end{lem}

\begin{proof}
  We have
\begin{multline*}
\prod_{\alpha \in wR^+} (1 - e^{\alpha})
= \prod_{\alpha \in wR^+ \cap R^+} (1 - e^{\alpha})
\prod_{\alpha \in wR^+ \cap (-R^+)} (1 - e^{\alpha})\\
= \prod_{\alpha \in wR^+ \cap R^+} (1 - e^{\alpha})
\prod_{\alpha \in (- wR^+) \cap R^+} (1 - e^{-\alpha}) \\
 = \epsilon(w) \prod_{\alpha \in (- wR^+) \cap R^+}e^{-\alpha}  
  \prod_{\alpha \in  R^+} (1 - e^{\alpha})  
  = \epsilon(w) e^{w \rho - \rho}
    \prod_{\alpha \in R^+} (1 - e^{\alpha}).
\end{multline*}
\end{proof}

Let $\mathfrak n_w$ be the nilpotent Lie algebra spanned by the root
subspaces corresponding to the roots in $w R^+$ for any $w \in W$.

\begin{lem}
  \label{antisym}
  Let $U$ be a Harish-Chandra module.  
  \begin{enumerate}
\item[(i)] For any $w \in W_0$, we have
  $$
  w \left(\sum_{ p \in \mathbb Z} (-1)^p [H_p(\mathfrak n,U)]\right)
  = \epsilon(w) \sum_{ p \in \mathbb Z} (-1)^p [H_p(\mathfrak n,U)]
    \cdot e^{w \rho - \rho} 
$$
    in $\K (\mathfrak t,T)$.
  \item[(ii)]
    For any $w \in W$, we have 
  $$
  \sum_{p \in \mathbb Z_+} (-1)^p [H_p(\mathfrak n,U)]
  = \epsilon(w) \sum_{p \in \mathbb Z_+} (-1)^p [H_p(\mathfrak n_w,U)] \cdot e^{\rho - w\rho}.
  $$
\end{enumerate}
\end{lem}

\begin{proof}
  The proof is based on the Osborne character formula.

(i) Since the character $\Theta_U$ is given by a function constant on the
conjugacy classes of regular elements, we see that $\Theta_U(t^w) =
\Theta_U(t)$ for any $t \in T_0'$ and $w \in W_0$.

By \ref{weyldenom}, we see that
\begin{multline*}
  \left(w \left(\sum_{ p \in \mathbb Z} (-1)^p \Theta_{H_p(\mathfrak
    n,U)}\right)\right)(t) = \Theta_U(t^{w^{-1}}) \prod_{\alpha \in
    R^+} (1 - e^{\alpha}(t^{w^{-1}})) \\
  = \Theta_U(t) \prod_{\alpha
    \in w R^+} (1 - e^{\alpha}(t))
  = \epsilon(w) e^{w \rho - \rho}(t)
  \Theta_U(t) \prod_{\alpha \in R^+} (1 - e^{\alpha}(t)) \\
  = \epsilon(w) e^{w \rho - \rho}(t)
  \sum_{ p \in \mathbb Z} (-1)^p \Theta_{H_p(\mathfrak n,U)}(t)
\end{multline*}
for any $t \in T_0$, and $(i)$ follows.

(ii) We can calculate $\Theta_U$ on $T_0'$ in two different ways
  $$
\Theta_U =  \frac { \sum_{p \in \mathbb Z} (-1)^p \Theta_{H_p(\mathfrak n,U)}}
      {\prod_{\alpha \in R^+} (1 - e^{\alpha})}
      =  \frac { \sum_{p \in \mathbb Z} (-1)^p \Theta_{H_p(\mathfrak n_w,U)}}
  {\prod_{\alpha \in wR^+} (1 - e^{\alpha})} .
  $$
  Therefore, we have
  $$
  \left( \sum_{p \in \mathbb Z} (-1)^p \Theta_{H_p(\mathfrak n,U)} \right)
  \left( {\prod_{\alpha \in wR^+} (1 - e^{\alpha})}\right)
  = \left(\sum_{p \in \mathbb Z} (-1)^p \Theta_{H_p(\mathfrak n_w,U)} \right) \left(
  {\prod_{\alpha \in R^+} (1 - e^{\alpha})}\right)
  $$
  on $T_0$. By \ref{weyldenom}, this implies
  $$
  \sum_{p \in \mathbb Z} (-1)^p \Theta_{H_p(\mathfrak n,U)}
    = \epsilon(w)e^{\rho - w\rho}
    \left( \sum_{p \in \mathbb Z} (-1)^p \Theta_{H_p(\mathfrak n_w,U)} \right).
  $$
\end{proof}

\subsection{Euler characteristic of Lie algebra homology of standard
Harish-Chandra modules} In this section we want to discuss the
formulas for the Euler characteristic of Lie algebra homology for
$\mathfrak n$ of standard Harish-Chandra modules $\Gamma(X,\mathcal
I(Q,\tau))$.

Since we are in the equal rank case, by \cite[5.9]{penrose}, an orbit
$Q$ is closed if and only if it is attached to the Cartan subalgebra
$\mathfrak t$. More precisely, any closed orbit $Q$ contains a Borel
subalgebra $\mathfrak b_w = \mathfrak t \oplus \mathfrak n_w$ for some
$w \in W$ and two such Borel subalgebras $\mathfrak b_u$ and
$\mathfrak b_v$ lie in the same orbit if and only if $u$ and $v$ are
in the same right $W_0$-coset in $W$.

Let $x_w$ be the point in the flag variety corresponding to the Borel
subalgebra $\mathfrak b_w$. As explained in \cite[p.~303]{hmsw1}, to
each $x_w$, one attaches a natural isomorphism of the dual $\mathfrak
h^*$ of the abstract Cartan algebra $\mathfrak h$ with $\mathfrak t^*$
which we call the {\em specialization} at $x_w$. Clearly, the
specializations at $x$ and $x_w$ differ by the action of $w$.

Assume first that $\lambda \in \mathfrak h^*$ is regular. Let $U$ be
a Harish-Chandra module in $\mathcal A(\mathcal U_\theta,K)$.
  
As we remarked before, Lie algebra homology groups $H_p(\mathfrak
n,U)$, $p \in \mathbb Z_+$, are finite-dimensonal representations of
$T$. Moreover, we have
$$
H_p(\mathfrak n,U) = \bigoplus_{w \in W} H_p(\mathfrak n,U)_{(w\lambda + \rho)}
$$
where $\mathfrak t$ acts on $H_p(\mathfrak n,U)_{(w\lambda +
  \rho)}$ via the specialization of $w\lambda+\rho$ \cite[Ch.~3,
  Cor. 2.4]{book}.

By \cite[Ch.~3, Cor.~2.6]{book}, the derived geometric fibers
$LT_{x_w}(\mathcal I(Q,\tau))$ of $\mathcal I(Q,\tau)$ at the point
$x_w$ correspond to $(\lambda + \rho)$-components of Lie algebra
homology for $\mathfrak n_w$ of $\Gamma(X,\mathcal I(Q,\tau))$ under
the specialization of $\lambda + \rho$ at $x_w$. Therefore, by
\ref{antisym}.(ii), calculating the $(\lambda + \rho)$-components of
Euler characteristic of Lie algebra homology for $\mathfrak n_w$ for
all $w \in W$, gives us the formula for Euler characteristic of Lie
algebra homology for $\mathfrak n$ of $\Gamma(X,\mathcal I(Q,\tau))$.

  First we consider the case where $Q$ is not closed in $X$.
  
\begin{lem}
\label{stdvan}
  Assume that the orbit $Q$ is not closed. Then
  $$
\sum_{p\in \mathbb Z} (-1)^p [H_p(\mathfrak n,\Gamma(X,\mathcal I(X,Q)))] = 0.
  $$
\end{lem}

\begin{proof}
Assume first that $\lambda$ is regular. Since $Q$ is not closed, the
points $x_w$, $ w \in W$, are not in $Q$.
  
Let $i_x : \{x\} \longrightarrow X$ and $i_Q : Q \longrightarrow X$ be
the natural inclusions. Since the standard Harish-Chandra sheaf $\mathcal
 I(Q,\tau)$ is the $D$-module direct image $i_{Q,+}(\tau)$, by the base
 change \cite[Ch.~VI, 8.5]{borel}, we see that $i_{x_w}^!(\mathcal
 I(Q,\tau)) = 0$, i.e., $LT_{x_w}(\mathcal I(Q,\tau)) = 0$.
By the above discussion, this implies that
$$
H_p(\mathfrak n_w, \Gamma(X,\mathcal I(Q,\tau)))_{(\lambda+\rho)} = 0
$$
for all $p \in \mathbb Z_+$. Hence, we have
$$
\sum_{p \in \mathbb Z}(-1)^p [H_p(\mathfrak n_w,
  \Gamma(X,\mathcal I(Q,\tau)))_{(\lambda+\rho)}] = 0
$$
for all $w \in W$. As we remarked, this immediately implies our
statement for regular $\lambda$. In particular the character of this
standard Harish-Chandra module vanishes on $T_0'$. Since coherent
continuation corresponds to twisting the localization by sections of a
homogeneous line bundle on $X$ followed by taking global sections
\cite[Ch.~3, Thm. 7.7]{book}, it follows that the character vanishes
on $T_0'$ also for singular $\lambda$. This in turn implies the
statement in general.
\end{proof}

Now we treat the case of closed orbits. We can pick $\mathfrak n$ so
that the corresponding point $x$ in the flag variety is in $Q$.
Denote by $j_x :\{x\} \longrightarrow Q$ the natural inclusion. Then
the geometric fiber $V = T_x(\tau)$ is an irreducible module in
$\mathcal M(\mathfrak t,T)$.

Since $Q$ is closed, $\mathfrak b \cap \mathfrak k$ is a Borel
subalgebra in $\mathfrak k$ and $s = \frac 1 2 \dim(K/T) = \dim Q$.

\begin{lem}
  \label{stdchar}
Assume that $Q$ is a closed orbit. Then we have
$$
\sum_{p \in \mathbb Z} (-1)^p [H_p (\mathfrak n,\Gamma(X,\mathcal I(Q,\tau)))]
= (-1)^{s + \Card R^+} \sum_{w \in W_0} \epsilon(w)
  [V]^w e^{\rho - w\rho}.
$$
\end{lem}

\begin{proof}
  Assume first that $\lambda$ is regular. By our assumption, the
  points $x_w$ are in $Q$ if and only if $w \in W_0$. As in the proof
  of \ref{stdvan} we conclude that $(\lambda + \rho)$-components of
  the Euler characteristic of Lie algebra homology for $\mathfrak n_w$
  (for the specialization of $\lambda + \rho$ at $x_w$) vanish for $w$
  outside $W_0$. This in turn implies that the $(w\lambda +
  \rho)$-components of Euler characteristic of Lie algebra homology
  for $\mathfrak n$ (for the specialization at $x$) vanish for $w$
  outside $W_0$.
  
Applying base change again \cite[Ch.~VI, 8.4]{borel}, we see that
 $$
 i_x^!(\mathcal I(Q,\tau)) = j_x^!(\tau) = T_x(\tau)[-\dim Q] = V[-\dim Q].
 $$
 Hence, we have $L T_x(\mathcal I(Q,\tau)) = D(V)[\codim Q]$. This
 immediately implies that
$$
  H_p(\mathfrak n,\Gamma(X,\mathcal I(Q,\tau)))_{(\lambda+\rho)} =
  \begin{cases}
    V, & \text{ if } p = \codim Q; \\
    0, & \text{ if } p \ne \codim Q.
    \end{cases}
$$
  
As we mentioned above, we have $\codim Q = \dim X - s$. Therefore, we
see that
$$
\sum_{p=0}^n  (-1)^p [H_p(\mathfrak n,\Gamma(X,\mathcal I(Q,\tau)))_{(\lambda+\rho)}]
 = (-1)^{\codim Q} [V] = (-1)^{s + \Card R^+} [V].
 $$
 The $(w \lambda + \rho)$-components of Euler characteristic of
 Lie algebra homology for $\mathfrak n$, for $w \in W_0$, are
 uniquely determined by \ref{antisym}.(i), i.e., we have
 $$
 \sum_{p=0}^n  (-1)^p [H_p(\mathfrak n,\Gamma(X,\mathcal I(Q,\tau)))]
 = (-1)^{s + \Card R^+} \sum_{w \in W_0} \epsilon(w) [V]^w e^{\rho - w \rho}.
 $$
This completes the proof for regular $\lambda$.

The reduction of the general case to the case of regular $\lambda$ is
the same as in the proof of \ref{stdvan}.
\end{proof}

\subsection{Homological pairing in the equal rank case}
As we remarked above the homological pairing $\langle [U] \mid
[\Gamma(X,\mathcal I(Q,\tau))\check{\ }] \rangle_{(\mathfrak g,K)}$
could be nonzero only if the second variable is a class attached to a
closed orbit $Q$.

Going back to our calculation of homological pairing in this
situation, by \eqref{calc1}, we have
\begin{multline}
  \label{calc2}
\langle [U] \mid [\Gamma(X,\mathcal I(Q,\tau)\check{\ }] \rangle_{(\mathfrak
  g,K)}
= \sum_{ p \in \mathbb Z} (-1)^{s+p} \langle [H_p(\mathfrak
  n,U)] \mid [V\check{\ } \otimes T_x(\Omega_X)]\rangle_{(\mathfrak t,T)} \\
= \sum_{ p \in \mathbb Z} (-1)^{s+p} \langle [H_p(\mathfrak
  n,U)] \mid [V\check{\ }] \cdot e^{2\rho}\rangle_{(\mathfrak t,T)}
\end{multline}
  
Now we want to rewrite the right side of \eqref{calc2} in a more
symmetric form. First, we have
$$
\langle [H_p(\mathfrak n,U)] \mid [V \check{\ }] \cdot e^{2\rho}
\rangle_{(\mathfrak t,T)} =
\langle [H_p(\mathfrak n,U)]^w \mid [V \check{\ }]^w
  \cdot e^{2w\rho} \rangle_{(\mathfrak t,T)}
$$
for any $w \in W_0$.

Hence, by summing over the group $W_0$, we get
$$
\langle [H_p(\mathfrak
  n,U)] \mid [V\check{\ }] \cdot e^{2\rho}\rangle_{(\mathfrak t,T)}
= \frac 1 {[W_0]} \sum_{w \in W_0} \langle [H_p(\mathfrak n,U)]^w \mid
[V\check{\ }]^w \cdot e^{2w\rho} \rangle_{(\mathfrak t,T)}.
$$

This implies, by \ref{antisym}.(i) and \eqref{calc2}, that
\begin{multline}
  \label{calc3}
\langle [U] \mid [\Gamma(X,\mathcal I(Q,\tau)\check{\ }] \rangle_{(\mathfrak
  g,K)} = \sum_{ p \in \mathbb Z} (-1)^{s+p} \langle [H_p(\mathfrak
  n,U)] \mid [V\check{\ }] \cdot e^{2\rho}\rangle_{(\mathfrak t,T)} \\ =
\sum_{ p \in \mathbb Z} (-1)^{s+p} \left( \frac 1 {[W_0]} \sum_{w \in W_0}
\big\langle [H_p(\mathfrak n,U)]^w \mid [V\check{\ }]^w \cdot e^{2w\rho} \big\rangle_{(\mathfrak t,T)} \right) \\
=
\frac {(-1)^s} {[W_0]}  \sum_{w \in W_0} \left( \sum_{ p \in \mathbb Z} (-1)^p
\big\langle [H_p(\mathfrak n,U)]^w \mid [V\check{\ }]^w \cdot e^{2w\rho} \big\rangle_{(\mathfrak t,T)} \right) \\
=
\frac {(-1)^s} {[W_0]} \sum_{w \in W_0} \Biggl\langle w \left(\sum_{ p \in \mathbb Z} (-1)^p [H_p(\mathfrak n,U)]\right) \Biggm|  [V\check{\ }]^w \cdot e^{2w\rho} \Biggr\rangle_{(\mathfrak t,T)} \\
=
\frac {(-1)^s} {[W_0]} \sum_{w \in W_0} \epsilon(w) \Biggl\langle \sum_{ p \in \mathbb Z} (-1)^p [H_p(\mathfrak n,U)] \cdot e^{w \rho - \rho} \Biggm| [V\check{\ }]^w \cdot e^{2w\rho}
\Biggr\rangle_{(\mathfrak t,T)} \\
= \frac {(-1)^s} {[W_0]}  \Biggr\langle \sum_{ p \in \mathbb Z} (-1)^p [H_p(\mathfrak n,U)] \Biggm| \sum_{w \in W_0} \epsilon(w) [V\check{\ }]^w \cdot  e^{\rho + w\rho} \Biggl\rangle_{(\mathfrak t,T)} .
\end{multline}

To complete our calculation, we need a representation theoretic
interpretation of the second sum in the above pairing.

The character of the dual representation
$\Gamma(X,\mathcal I(Q,\tau))\check{\ }$ satisfies
$\Theta_{\Gamma(X,\mathcal I(Q,\tau))\check{\ }} (t) = \Theta_{\Gamma(X,\mathcal I(Q,\tau))}(t^{-1})$ for any $t \in T_0'$.
Hence, by \ref{stdchar} and the Osborne conjecture, we have
\begin{multline*}
  \Theta_{\Gamma(X,\mathcal I(Q,\tau))\check{\ }} = (-1)^{s + \Card R^+}
  \frac{ \sum_{w \in W_0} \epsilon(w)
    \Theta_{[V\check{\ }]^w} e^{w\rho - \rho}} {\prod_{\alpha \in R^+} (1 - e^{-\alpha})} \\
= (-1)^s
  \frac{ \sum_{w \in W_0} \epsilon(w)
    \Theta_{[V\check{\ }]^w} e^{w\rho + \rho}} {\prod_{\alpha \in R^+} (1 - e^{\alpha})}
\end{multline*}
on $T_0'$. Moreover, using the Osborne conjecture again, we have
$$
  \sum_{p=0}^n (-1)^p [H_p(\mathfrak n, \Gamma(X,\mathcal I(Q,\tau))\check{\ })] =
  (-1)^s \sum_{w \in W_0} \epsilon(w) [V\check{\ }]^w \cdot e^{\rho + w\rho}.
$$

Plugging this into \eqref{calc3}, for any Harish-Chandra module $U$, we
finally get
\begin{multline}
  \label{calc4}
  \langle [U] \mid [\Gamma(X,\mathcal I(Q,\tau))\check{\ }]
  \rangle_{(\mathfrak g,K)} \\
  = \frac 1 {[W_0]} \sum_{ p \in \mathbb Z} (-1)^{p+q}
  \langle H_p(\mathfrak n,U) \mid
  H_q(\mathfrak n,\Gamma(X,\mathcal I(Q,\tau))\check{\ })\rangle_{(\mathfrak t,T)}.
\end{multline}

In the calculation leading to \eqref{calc4} the choice of the Lie
algebra $\mathfrak n$ (or equivalently set of positive roots $R^+$)
was specific for the orbit $Q$. On the other hand, the symmetry
established in \ref{antisym}.(ii), implies that the formula holds for
any set of positive roots $wR^+$ in $R$. Therefore, the formula holds
for any $\mathfrak n$ attached to the Cartan subalgebra $\mathfrak t$.

This finally leads to the following result which completely determines
the homological pairing on $\K (\mathfrak g, K)$.

\begin{thm}
  \label{homo}
  Let $\mathfrak t$ be a Cartan subalgebra of $\mathfrak g$ contained
  in $\mathfrak k$ and $R^+$ a set of positive roots in the root
  system $R$ of $(\mathfrak g,\mathfrak t)$. Let $\mathfrak n$ be the
  Lie algebra spanned by root subspaces corresponding to $R^+$. Then
  we have
  $$
  \langle [U] \mid [U'] \rangle_{(\mathfrak g,K)} = \frac 1 {[W_0]} \sum_{ p,q \in \mathbb Z}
  (-1)^{p+q} \dim \Hom_{(\mathfrak t,T)}(H_p(\mathfrak n,U),H_q(\mathfrak n,U')) 
  $$
  for any two Harish-Chandra modules $U$ and $U'$.
\end{thm}

\begin{proof}
  To establish the above formula, it is enough to check it on duals
  $\Gamma(X,\mathcal I(Q,\tau))\check{\ }$ of standard Harish-Chandra
  modules in the second variable for any geometric data $(Q,\tau)$.

If $Q$ is closed, the formula is established in \eqref{calc4}.

Otherwise, if $Q$ is not closed, by \ref{stdvan} and the Osborne
conjecture, the character of the standard module $\Gamma(X,\mathcal
I(Q,\tau))$ vanishes on $T_0'$. This in turn implies that the
character of its dual vanishes on $T_0'$. Hence, we have
  $$
  \sum_{p \in \mathbb Z} (-1)^p [H_p(\mathfrak n,
    \Gamma(X,\mathcal I(Q,\tau))\check{\ })] = 0
  $$ and the right side of the above formula vanishes. Since a
  $\sigma$-stable Cartan subalgebra attached to $Q$ has $\mathfrak a
  \ne \{0\}$, the left side vanishes by \eqref{calc1} (as we already
  remarked).
\end{proof}

\section{Proof of the analogue of Kazhdan's conjecture}
\subsection{Equality of two pairings}
Finally, we prove our main result.

\begin{thm}
The elliptic and homological pairing on $\K (\mathfrak g,K)$ agree.
\end{thm}

\begin{proof}
  In the case of nonequal rank, the elliptic pairing on $\K(\mathfrak
  g,K)$ is zero by definition. The homological pairing is zero by \ref{van2}.

 Assume that the $\rank G_0 = \rank K_0$. Then, by \ref{homo} and the
 orthogonality relations for the compact group $T_0$, for any two
 Harish-Chandra modules $U$ and $U'$ we have
 $$
 \langle [U] \mid [U'] \rangle_{(\mathfrak g,K)} = \frac 1 {[W_0]} \sum_{ p,q \in \mathbb Z}
 (-1)^{p+q} \int_{T_0} \Theta_{H_p(\mathfrak n,U)}(t)  \overline{\Theta_{H_q(\mathfrak n,U')}(t)} \, d\mu_{T_0}(t).
 $$
 By the Osborne conjecture, we see that
 $$
 \langle [U] \mid [U'] \rangle_{(\mathfrak g,K)} = \frac 1 {[W_0]} \int_{T_0} \Psi_U(t)  \overline{\Psi_{U'}(t)} \, d\mu_{T_0}(t) =  \langle [U] \mid [U'] \rangle_{ell}.
 $$
\end{proof}

\bibliographystyle{amsplain}
\newcommand{\noopsort}[1]{}
\providecommand{\bysame}{\leavevmode\hbox to3em{\hrulefill}\thinspace}
\providecommand{\MR}{\relax\ifhmode\unskip\space\fi MR }
\providecommand{\MRhref}[2]{%
  \href{http://www.ams.org/mathscinet-getitem?mr=#1}{#2}
}
\providecommand{\href}[2]{#2}

\end{document}